\documentclass[11pt]{amsart}

\usepackage[utf8]{inputenc}
\usepackage{a4wide}
\usepackage{color}
\usepackage{amsfonts,amsmath,amsthm,amssymb}
\usepackage{verbatim}
\usepackage{epsfig}
\usepackage{graphics}
\usepackage{mathtools}
\usepackage{array}
\usepackage{enumerate}
\usepackage[ruled]{algorithm}
\usepackage{algorithmic}
\usepackage{bbm}
\usepackage[font=small,labelfont=bf]{caption}
\usepackage{wrapfig}
\usepackage{subcaption}
\usepackage[all, knot, cmtip]{xy} 
\usepackage{wrapfig}

\newcommand{\ndN}{\mathbb{N}}

\renewcommand{\Pr}[1]{\mathbb{P}(#1)}
\newcommand{\Ex}[1]{\mathbb{E}[#1]}

\newcommand{\one}{\mathbbm{1}}
\newcommand{\w}{w}

\newcommand{\cC}{\mathcal{C}}
\newcommand{\cM}{\mathcal{M}}

\newcommand{\cT}{\mathcal{T}}
\newcommand{\cR}{\mathcal{R}}
\newcommand{\cA}{\mathcal{A}}
\newcommand{\cH}{\mathcal{H}}

\newcommand{\mT}{\mathsf{T}}

\newcommand{\me}{\mathsf{e}}

\newcommand{\CRT}{\mathcal{T}_\me}

\newcommand{\cE}{\mathcal{E}}

\newcommand{\He}{\text{H}}
\newcommand{\Di}{\text{D}}

\newcommand{\he}{\text{h}}

\newcommand{\spa}{\mathsf{span}}

\newcommand{\eqdist}{\,{\buildrel (d) \over =}\,}

\newcommand{\convdis}{\,{\buildrel (d) \over \longrightarrow}\,}
\newcommand{\convp}{\,{\buildrel p \over \longrightarrow}\,}

\newcommand{\Seq}{\textsc{SEQ}}

\newtheorem{theorem}{Theorem}[section]

\newtheorem{corollary}[theorem]{Corollary}

\newtheorem{lemma}[theorem]{Lemma}

\numberwithin{equation}{section}

\title{Scaling limits of random outerplanar maps with independent link-weights}

\author{Benedikt Stufler}
\address[Benedikt Stufler]{\'Ecole Normale Sup\'erieure de Lyon}
\email{benedikt.stufler@ens-lyon.fr}

\begin{document}

\begin{abstract}
The scaling limit of large simple outerplanar maps was established by Caraceni using a bijection due to Bonichon, Gavoille and Hanusse. The present paper introduces a new bijection between outerplanar maps and trees decorated with ordered sequences of edge-rooted dissections of polygons. We apply this decomposition in order to provide a new, short proof of the scaling limit that also applies to the general setting of first-passage percolation. We obtain sharp tail-bounds for the diameter and recover the asymptotic enumeration formula for outerplanar maps. Our methods also enable us to treat subclasses such as bipartite outerplanar maps.
\end{abstract}
\maketitle

\section{Introduction and main results}
The continuum random tree (CRT) was constructed by Aldous \cite{MR1085326,MR1166406,MR1207226} and shown to be the scaling limit of several models of random trees. Since then, the study of scaling limits of random discrete structures such as trees, graphs and planar maps has developed into a very active field with contributions by a wide variety of researchers \cite{MR3382675,MR3050512,MR3335010,MR3342658,2015arXiv150207180P}.

Much of this progress was made possible by the use of appropriate combinatorial bijections that relate these objects to trees endowed with additional structures such as vertex colourings. The reason for this is that trees are generally easier to analyse and such bijections allow for a transfer of results for random trees to the objects under consideration.

The present paper concerns itself with rooted simple outerplanar maps. These planar maps may be encoded as bicolored trees of a certain class by using a bijection due to Bonichon, Gavoille and Hanusse \cite{MR2185278}. Their scaling limit was established by Caraceni \cite{Ca} using this bijection and relating the geodesics in the trees and planar maps:

\begin{theorem}[{\cite[Thm. 1.1]{Ca}}]
\label{te:main}
Let $\mathbf{M}_n$ be the random map drawn uniformly among all rooted maps with $n$ vertices that are simple and outerplanar. As $n$ becomes large,
\[
 (\mathbf{M}_n, \frac{9}{7 \sqrt{2}} n^{-1/2} d_{\mathbf{M}_n}) \convdis (\CRT, d_{\CRT})
\]
in the Gromov-Hausdorff sense. Here $\CRT$ denotes the continuum-random tree constructed from Brownian excursion.
\end{theorem}

In the following we introduce a new bijective decomposition that identifies (rooted) simple outerplanar maps as a certain class of trees decorated with ordered sequences of dissections of edge-rooted polygons.
This allows us to provide a new proof of Theorem~\ref{te:main} that is short and extends the result in two directions. First, we may treat in a unified way subclasses of outerplanar maps that are stable under taking non-separable components, for example bipartite outerplanar maps. Second, we prove a scaling limit for the more general first-passage percolation metric $d_{\text{FPP}}$ obtained by assigning an independent random positive weight to each edge and letting the distance of two vertices be the minimum of sums of weights along any joining paths. Here we allow unbounded link weights, but do require finite exponential moments. Of course, this includes the classical case of the graph-metric. We obtain sharp exponential tail-bounds for the $d_{\text{FPP}}$-diameter of random outerplanar maps and precise asymptotic expressions for its moments. Studying first-passage percolation on random planar maps has received some attention in recent literature, see for example \cite{2014arXiv1412.5509C,2014arXiv1408.3040A,CuLG}. We also apply our decomposition to recover the asymptotic enumeration formula for outerplanar maps given in \cite{MR2185278} and obtain a similar formula for the bipartite case.

Let us make this precise. Recall that a {\em planar map} is a 2-cell embedding of a connected planar multigraph on the sphere, considered up to orientation-preserving homeomorphism. If one of the edges is distinguished and given an orientation, then the map is termed {\em rooted}. This oriented edge is called the {\em root edge} of the map and its origin is termed the {\em root vertex}. We call the face to the left of the root edge the {\em root face} and the face to the right the {\em outer face}. The outer face is taken as the infinite face in plane representations. By convention, we also consider the map consisting of a single vertex as rooted, although it has no edges to be rooted at. We say a map is {\em simple}, if it has no loops nor multiple edges. Finally, recall that a map is termed {\em non-separable}, if it has at least one edge and removing any vertex does not disconnect the map.
A simple rooted maps is termed {\em outerplanar} if every vertex lies on the boundary of the outer face.

In order to describe the subclasses under consideration, suppose that we are given a non-empty class $\cC^{s}$ of non-separable rooted outerplanar maps, i.e. a set of dissections of edge-rooted polygons. We may form the class $\cM^s$ of all (rooted and simple) outerplanar maps whose maximal non-separable submaps are required to belong to $\cC^{s}$. For example, in the case of bipartite outerplanar maps,  $\cC^{s}$ is given by the unique simple map with $2$ vertices and all dissections of (edge-rooted) polygons in which each face has even degree.

Roughly speaking, we will restrict ourself to subclasses of outerplanar maps having the property, that all non-separable submaps of a typical large map are small compared to the total number of vertices. In order to describe this requirement formally, we introduce the following notation. Let $\varphi(z)$ denote the power series such that the coefficient $[z^k]\varphi(z)$ of $z^k$ in $\varphi(z)$ is given by the number of maps in the class $\cC^s$ with $k$ vertices. Set $\phi(z) = 1/(1-\varphi(z)/z)$ and let $\rho_\phi$ denote the corresponding radius of convergence. Finally, set $\nu = \lim_{t \nearrow \rho_\phi} \psi(t) \in [0, \infty]$ with $\psi(t) = t \phi'(t)/\phi(t)$ and let $\mathbf{s}$ denote the greatest common divisor of all integers $i$ with $[z^i]\phi(z) \ne 0$. 

\begin{theorem} 
\label{te:main2}
Let $\mathbf{M}_n^s$ be the uniformly at random drawn map from the subclass $\cM^s$ with $n$ vertices. Consider the first passage percolation metric $d_{\text{FPP}}$ on $\mathbf{M}_n^s$ in which each edge receives an independent copy of a random positive weight having finite exponential moments.
If $\nu > 1$, then there exists a constant $\kappa >0$ such that
\[
 (\mathbf{M}_n^s, \kappa n^{-1/2} d_{\text{FPP}}) \convdis (\CRT, d_{\CRT})
\]
with respect to the Gromov-Hausdorff topology, as $n \equiv 1 \mod \mathbf{s}$ becomes large. Morever, there are constants $C,c>0$ such that for all $x\ge0$ and $n$ we have the following tailbound for the diameter
\[
\Pr{\Di_{\text{FPP}}(\mathbf{M}_n^s) \ge x} \le C (\exp(-cx^2/n) + \exp(-cx)).
\]
This applies to unrestricted and bipartite outerplanar maps, as in these cases we have $\nu = \infty$. 
\end{theorem}
Note that if the link-weights are bounded, then $\Di_{\text{FPP}}(\mathbf{M}_n^s)$ is bounded by a constant multiple of $n$ and hence the tail-bound may be simplified to
\[
\Pr{\Di_{\text{FPP}}(\mathbf{M}_n^s) \ge x} \le D \exp(-d x^2/n).
\]

The constant $\nu$ has a natural interpretation in terms of simply generated trees. Unless $\rho_\phi=0$ (which is equivalent to $\nu=0$), $\nu$ is the supremum of the means of all probability weight sequences equivalent to the sequence of coefficients of the series $\phi(z)$. See Section 4 and in particular Remark 4.3 of Janson's survey \cite{MR2908619} for details.

We calculate the scaling constants for the graph-metric case (i.e. each edge receives weight $1$) for unrestricted and bipartite outerplanar maps, recovering Theorem~\ref{te:main} and obtaining:
\begin{theorem}
\label{te:main3}
Let $\mathbf{M}_n^{\text{bip}}$ denote the uniformly at random drawn (simple and rooted) bipartite outerplanar map with $n$ vertices. Then, in the Gromov-Hausdorff sense,
\[
(\mathbf{M}_n^{\text{bip}}, 36{\frac {\sqrt {\sqrt {3}-1}}{69-7\sqrt {3}}} n^{-1/2} d_{\mathbf{M}_n^{\text{bip}}}) \convdis (\CRT, d_{\CRT}).
\]
\end{theorem}
The scaling limit of arbitrary (i.e. not necessarily outerplanar) bipartite planar maps was established by Abraham \cite{2013arXiv1312.5959A} and is given by the Brownian map rather than the CRT.  The convergence towards the CRT implies that, under the assumptions of Theorem~\ref{te:main2}, we have for every fixed $r$
\[
\kappa^r n^{-r/2} \Di_{\text{FPP}}(\mathbf{M}_n^s)^r \convdis \Di(\CRT)^r.
\]
The exponential tail-bounds for the diameter ensure that $\Di_{\text{FPP}}(\mathbf{M}_n^s)$ is arbitrarily high uniformly integrable and consequently 
\[
\Ex{\Di_{\text{FPP}}(\mathbf{M}_n^s)^r} \sim n^{r/2} \kappa^{-r} \Ex{\Di(\CRT)^r}.
\]
The distribution of the diameter $\Di(\CRT)$ and its moments are known, see Section~\ref{sec:dicrt} below.

The bijective encoding of subclasses of outerplanar maps established in the present paper may also be used to asymptotically count these maps by their number of vertices.
\begin{theorem}
\label{te:main4}
If $\nu\ge1$, then the number $|\cM^s_n|$of planar maps in the class $\cM^s$ with $n$ vertices is asymptotically given by
\[
|\cM^s_n| \sim \mathbf{s} (\phi(\tau) / (2 \pi \phi''(\tau)))^{1/2} (\tau / \phi(\tau))^{-n}  n^{-3/2}
\]
as $n \equiv 1 \mod \mathbf{s}$ becomes large. Here $\tau$ denotes the unique solution of the equation $\psi(\tau)=1$ in the interval $]0, \rho_\phi]$.
\end{theorem}

As an application, we recover the asymptotic enumeration formula for outerplanar maps found by Bonichon, Gavoille and Hanusse \cite[Thm. 3]{MR2185278}, and establish a similar formula for the bipartite case.
\begin{corollary}
\label{co:enum}
The numbers $|\cM_n^{\text{out}}|$ and $|\cM_n^{\text{bip}}|$ of (bipartite) rooted simple outerplanar maps with $n$ vertices satisfy the asymptotics
\[
|\cM_n^{\text{out}}| \sim 8^n n^{-3/2} / (36 \sqrt{\pi}) \quad \text{and} \quad |\cM_n^{\text{bip}}| \sim {\frac { \left( 2\sqrt {3} -3\right) \sqrt {2}}{9\sqrt {\pi 
 \left( \sqrt {3}-1 \right) }}} (3 \sqrt{3}-5)^{-n} n^{-3/2}.
\]
\end{corollary}

\subsection{Remarks on the diameter of the CRT}
\label{sec:dicrt}
The distribution and moments of the diameter of the continuum random tree $\CRT$ are known and given by
\begin{align}
\label{eq:a}
\Pr{\Di(\CRT) > x} =
\sum_{k=1}^\infty (k^2-1)\Big(\frac{2}{3}k^4x^4 -4k^2x^2 +2\Big)\exp(-k^2x^2/2)
\end{align}
and
\begin{align}
\label{eq:b}
\Ex{\Di(\CRT)} &= \frac{4}{3}\sqrt{\frac{\pi}{2}}, \quad \Ex{\Di(\CRT)^2} = \frac{2}{3}\left(1 + \frac{\pi^2}{3}\right), \quad \Ex{\Di(\CRT)^3} = 2 \sqrt{2\pi}, \\
\label{eq:c}
\Ex{\Di(\CRT)^r} &= \frac{2^{r/2}}{3} r(r-1)(r-3) \Gamma(r/2)(\zeta(r-2) - \zeta(r)) \quad \text{if $r \ge 4$}.
\end{align}
The distribution of the diameter $\Di(\CRT)$ and its first moment $\Ex{\Di(\CRT)} = 4/3 \sqrt{\pi/2}$ have been known since the construction of the CRT by Aldous \cite[Ch. 3.4]{MR1166406}, who used the convergence of random labelled trees to the CRT together with results by Szekeres \cite{MR731595} regarding the diameter of these trees. Expression~(\ref{eq:a}) was also recovered directly in the continuous setting by Wang \cite{2015arXiv150305014W}.

The higher moments could be obtained directly from (\ref{eq:a}) by tedious calculations, or more easily by building on known results regarding random trees: Broutin and Flajolet studied in \cite{MR2956055} the random tree $\mT_n$ that is drawn uniformly at random among all unlabelled trees with $n$ leaves in which each inner vertex is required to have degree $3$. Using analytic methods \cite[Thm. 8]{MR2956055}, they computed asymptotics of the form \[\Ex{\Di(\mT_n)^r} \sim c_r \lambda^{-r} n^{r/2}\] with $\lambda$ an analytically given constant the constants $c_r$ given by
\begin{align*}
c_1 &= \frac{8}{3} \sqrt{\pi}, \quad c_2 = \frac{16}{3}(1 + \frac{\pi^2}{3}), \quad c_3 = 64 \sqrt{\pi}, \\ 
c_r &= \frac{4^r}{3}r(r-1)(r-3)\Gamma(r/2)(\zeta(r-2) - \zeta(r)) \quad \text{if $r \ge 4$}.
\end{align*} 
By recent results of the author \cite[Thm. 1.1]{2014arXiv1412.6333S} there is a constant $g$ such that the rescaled tree $g n^{-1/2} \mT_n$ converges in the Gromov-Hausdorff sense towards the CRT and there are constants $c,C>0$ with
$\Pr{\Di(\mT_n) > x} \le C \exp(-cx^2/n).$ Thus \[\Ex{\Di(\mT_n)^r} \sim \Ex{\Di(\CRT)^r} g^{-r}  n^{r/2} \] and hence 
\[
\Ex{\Di(\CRT)^r} = c_r (g/\lambda)^{r}.\] It remains to calculate the ratio $g/\lambda$, which is given by \[g/\lambda = \Ex{\Di(\CRT)} / c_1 = 1/(2 \sqrt{2}),\] since $\Ex{\Di(\CRT)} = 4/3 \sqrt{\pi/2}$. This yields the expressions in (\ref{eq:b}) and (\ref{eq:c}).

\section{A bijection between outerplanar maps and decorated trees}
The bijective encoding of outerplanar maps we are going to describe is best treated using the language of analytic combinatorics by Flajolet and Sedgewick \cite{MR2483235}, or combinatorial species by Joyal \cite{MR633783}. 

\begin{figure}[ht]
	\centering
	\begin{minipage}{0.8\textwidth}
		\centering
		\includegraphics[width=0.9\textwidth]{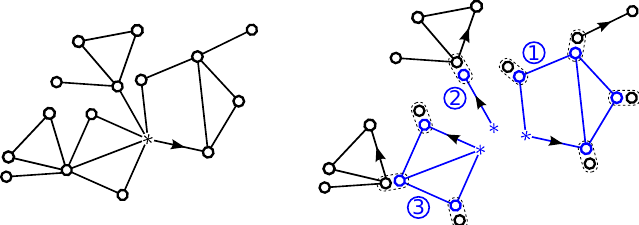}
		\caption{The decomposition of $\cM_*^{\text{out}}$-objects.}
		\label{fi:decore2}
	\end{minipage}
\end{figure}  

Let $\cM^{\text{out}}$ denote the class of rooted simple outerplanar maps and $\cC^{\text{out}}$ the class of rooted non-separable simple outerplanar maps, both times with vertices as atoms. Moreover, let $\cM_*^{\text{out}}$ and $\cC_*^{\text{out}}$ denote the corresponding classes in which the root vertex does not contribute to the total number of vertices (and is hence depicted as a $*$-vertex in illustrations).

In a similar way as for graphs (see for example Chapter 3.1 in Diestel's book \cite{MR2744811}), we may call a submap of a planar map $M$ a {\em block}, if it is non-separable and inclusion maximal with this property. Any two distinct blocks may intersect in at most one vertex, because otherwise their union would also be non-separable. Moreover, each edge of $M$ is a non-separable submap and hence contained in a unique block. Any non-separable simple outerplanar map with at least $3$ vertices has a unique Hamilton cycle given by the boundary of the outer face. Hence the class $\cC^{\text{out}}$ consists of edge-rooted dissections of polygons and the map consisting of two vertices connected by a root edge.

Let $M \in \cM_*^{\text{out}}$ be a given simple outerplanar map with root edge $e_1$. As we assumed $M$ to be outerplanar and simple, the counter-clockwise ordered list $e_1, \ldots, e_d$ of edges incident to the root vertex has the property, that if $e_i$ and $e_j$ are contained in the same block and $i<j$, then the edges $e_i, e_{i+1}, \ldots, e_{j}$ are also contained in this block. Consequently, the blocks of $M$ that contain the root vertex may be ordered in a natural way, yielding a sequence of elements $(C_1, \ldots, C_t)$ in $\cC^{\text{out}}_*$ as illustrated in Figure~\ref{fi:decore2}: If $M$ consists only of a single vertex, then the list is empty. Otherwise, it starts with the unique block containing the root edge of $M$. This block inherits the root edge of $M$ as its own root edge and hence may be considered as an element of $\cC^{\text{out}}_*$.  If $C_1$ already contains all edges incident to the root vertex, then $t=1$ and the list is complete. Otherwise, we may select the first edge after the edges belonging to $C_1$, orient it as pointing away from the root vertex, and let $C_2$ be the unique block rooted at this edge. The remaining blocks are selected and rooted in the same manner, until no edges incident to the root vertex are left.  

We let $\cR$ denote the class of ordered finite sequences of maps in $\cC^{\text{out}}_*$. The size of an $\cR$-object is the sum of the sizes of the individual planar maps, without counting the root vertices. Hence the blocks incident to the root vertex of $M$ may be interpreted as an $\cR$-object $R$. The non-root vertices of any map in $\cC_*^{\text{out}}$ may be ordered in a canonical way, by starting with the vertex to which the root edge points and continuing in a counter-clockwise way. Consequently, the non-root vertices of any $\cR$-object may also be ordered in a canonical way by concatenating the individual linear orders. 

Let $r$ denote the size of the $\cR$-object $R$. If we delete the root vertex of $M$ and all the edges of the blocks incident to the root vertex, we are left with a sequence $M_1, \ldots, M_r$ of submaps of $M$, such that for each $1 \le j \le r$ the map $M_j$ intersects $R$ only at its $j$th vertex $v_j$. As illustrated in Figure~\ref{fi:decore2}, each of these submaps may be rooted at an oriented edge in a natural way: For each $j$, we may consider the counter-clockwise ordered list of edges incident to $v_j$ in $M$, select the first that comes after the edges in $C_j$, and orient it as pointing away from $v_j$. That is, unless $M_j$ consists only of a single vertex, in which case we consider $M_j$ as edge-rooted by convention.

Summing up, we have established that any $\cM_*^{\text{out}}$-object may be decomposed into an ordered (possibly empty) sequence of $\cC_*^{\text{out}}$-objects, where at each non-$*$-vertex a $\cM^{\text{out}}$-object is inserted. Expressing this as a combinatorial specification yields
\[
\cM_*^{\text{out}}(z) = \Seq(\cC_*^{\text{out}}(\cM^{\text{out}}(z)))
\]
with the variable $z$ marking the number of vertices. Since the classes $\cM^{\text{out}}$ and $\cM^{\text{out}}_*$ are related by \[\cM^{\text{out}}(z) = z\cM_*^{\text{out}}(z),\] this may be expressed as a recursive decomposition
\begin{align}
\label{eq:co1}
\cM^{\text{out}}(z) = z \cR(\cM^{\text{out}}(z)),
\end{align}
with the class $\cR$ satisfying
\begin{align}
\label{eq:co2}
\cR(z) = \Seq(\cC_*^{\text{out}}(z)).
\end{align}

\begin{figure}[t]
	\centering
	\begin{minipage}{1.0\textwidth}
		\centering
		\includegraphics[width=1.0\textwidth]{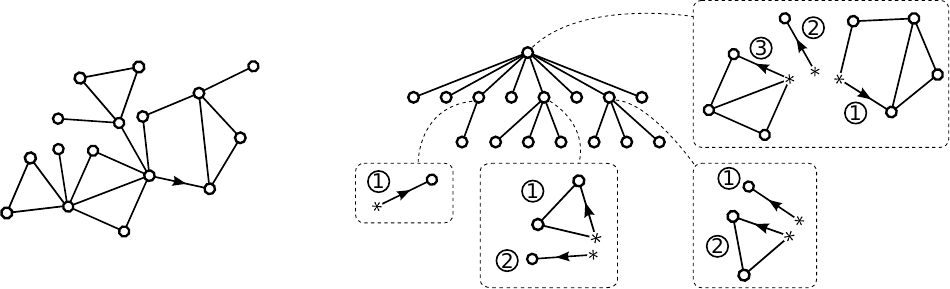}
		\caption{The decomposition of simple outerplanar rooted maps into decorated trees.}
		\label{fi:decompnew2}
	\end{minipage}
\end{figure}

Let $\cA$ denote the class of all pairs $(T,\alpha)$ with $T$ a plane tree and $\alpha$ a function that assigns to each vertex $v$ of $T$ an $\cR$-object $\alpha(v)$ whose size equals the out-degree $d_A^+(v)$ of the vertex. We are going to construct a size-preserving bijection between $\cM^{\text{out}}$ and $\cA$ by unwinding the recursive decomposition \eqref{eq:co1} as illustrated in Figure~\ref{fi:decompnew2}.

For each $M \in \cM^{\text{out}}$ the corresponding decorated tree $\Xi(M) = (T, \alpha)$ is assembled as follows. According to the decomposition \eqref{eq:co1}, the map $M$ corresponds to an $\cR$-object $R$ where $\cM^{\text{out}}$-objects $M_1, \ldots, M_t$ are inserted at each of its canonically ordered non-root vertices. We begin the construction by letting $T$ be a plane tree consisting of a root vertex $u$ and $t$ sons $u_1, \ldots, u_t$, and setting $\alpha(u) = R$. If $t=0$, then the construction is complete, and the number of vertices of $T$ equals the number of vertices of the map $M$. If $t \ge 1$, then for each $1 \le j \le t$ we may again decompose the map $M_j$ into an $\cR$-object $R_j$ where $\cM^{\text{out}}$-objects $M_{j,1}, \ldots, M_{j, t_j}$ are inserted into each of its ordered non-root vertices. For each $j$, we attach $t_j$ sons to the vertex $u_j$ and set $\alpha(u_j) = R_j$. The full tree $(T, \alpha)$ is then constructed by proceeding in this way, until we have explored the whole map. In each step we explore the same amount of new vertices as we add to the tree. Hence the number of vertices of the map $M$ equals the number of vertices of the tree $T$. We obtain a size-preserving function $\Xi: \cM^{\text{out}} \to \cA$.

The inverse function of $\Xi$ is constructed as follows. As illustrated in Figure \ref{fi:robj}, any $\cR$-object $R$ corresponds to a single planar map $M(R)$ from $\cM_*^\text{out}$, constructed by placing the individual maps in a counter-clockwise manner and gluing their root vertices together. The root edge of the first $\cC^{\text{out}}_*$-object becomes the root edge of the resulting map. The map $M$ corresponding to a decorated tree $(T, \alpha) \in \cA$ is constructed by taking the maps $(M(\alpha(v)))_{v \in T}$, and identifying for each vertex $v\in T$ and each offspring $w$ of $v$ the root vertex of $M(\alpha(w))$ with the corresponding vertex in $M(\alpha(v))$. We have thus established the following result.

\begin{figure}[t]
	\centering
	\begin{minipage}{1.0\textwidth}
		\centering
		\includegraphics[width=0.65\textwidth]{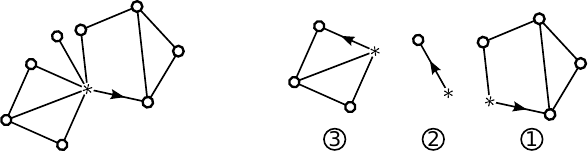}
		\caption{Construction of a map by gluing together the root vertices.}
		\label{fi:robj}
	\end{minipage}
\end{figure}

\begin{theorem}
\label{te:decomp}
The function $\Xi: \cM^{\text{out}} \to \cA$ is a bijection between the class of simple outerplanar maps and the set of all pairs $(T, \alpha)$ with $T$ a plane tree and $\alpha$ a function that assigns to each vertex $v\in T$ an $\cR$-structure with size $d^+_T(v)$. Here maps with $n$ vertices correspond to decorated trees with $n$ vertices.
\end{theorem}

The bijection of Theorem~\ref{te:decomp} is illustrated in Figure~\ref{fi:decompnew2}.
Given a subclass $\cC^s \subset \cC^{\text{out}}$ of non-separable maps we may form the subclass $\cM^s \subset \cM^{\text{out}}$ of all maps whose (canonically rooted) non-separable submaps are required to be elements of $\cC^s$. For example, bipartite outerplanar maps fall into this setting, for which the corresponding class $\cC^{\text{bip}}$ of non-separable maps is given by all bipartite dissections of polygons with an even number of vertices. We define $\cR^s$ as the class of all finite sequences of $\cC^s$-objects. That is, 
\begin{align}
\cR^s(z) = \Seq(\cC_*^s(z)).
\end{align}
The arguments of this section may easily be repeated to obtain 
\begin{align}
\cM^s(z) = z \cR^s(\cM^s(z)).
\end{align}
Consequently, maps from $\cM^s$ correspond to plane trees decorated with $\cR^s$-objects:
\begin{lemma}
\label{le:decomp}
The restriction $\Xi \mid_{\cM^s}$ of the map in Theorem~\ref{te:decomp} is a size-preserving bijection between the subclass $\cM^s \subset \cM^{\text{out}}$ and the subclass $\cA^s \subset \cA$ of all pairs $(T, \alpha) \in \cA$ satisfying $\alpha(v) \in \cC^s$ for all $v \in T$.
\end{lemma}

\section{Proofs of Theorems~\ref{te:main2} and \ref{te:main4}}

\subsection{Sampling and counting outerplanar maps}
For any integer $k$ let $\omega_k$ denote the number of  $\cR^s$-objects with $k$ non-$*$-vertices. This defines a {\em weight-sequence} $\mathbf{w} = (\omega_k)_k$. To any plane tree $T$ we may assign the weight
\[
\omega(T) = \prod_{v \in T} \omega_{d^+_T(v)}
\]
with $d^+_T(v)$ denoting the outdegree of a vertex $v$. A {\em simply generated tree} $\cT_n$ with weight-sequence $\mathbf{w}$ is a random plane tree with $n$ vertices such that any tree gets drawn with probability proportional to its weight. The sum $Z_n$ of all weights of plane trees with $n$ vertices is called the {\em partition function}.
\begin{lemma}
\label{le:sampler}
Let $n\in \ndN$ be an integer with $Z_n \ne 0$. Then the following procedure draws a random outerplanar simple rooted map with $n$ vertices from the class $\cM^s$ uniformly at random.
\begin{enumerate}[1.]
\item Let $\cT_n$ denote a simply generated tree with weight sequence $\mathbf{w}$.
\item For any vertex $v$ draw $\beta_n(v)$ uniformly at random from all $\cR^s$-structures  with size $d_{\cT_n}^+(v)$.
\item Apply the bijection of Theorem~\ref{te:decomp} to the decorated tree $(\cT_n, \beta_n)$ in order to obtain a rooted planar map.
\end{enumerate}
\end{lemma}
\begin{proof}
	For any decorated plane tree $(T, \beta)$ with $n$ vertices it holds that
	\begin{align*}
	\Pr{(\cT_n, \beta_n) = (T, \beta)} &= Z_n^{-1} \omega(T)   \Pr{\beta_n = \beta \mid \cT_n = \cT} 
	= Z_n^{-1} \omega(T)  \prod_{v \in T} \omega_{d^+_T(v)}^{-1} 
	= Z_n^{-1}.
	\end{align*}
	Hence $(\cT_n, \beta_n)$ is uniformly distributed among all $\cR^s$-decorated plane trees with $n$ vertices.
\end{proof}
Note that the proof above also shows that the partition function $Z_n$ counts the number of maps from the class $\cM^s$ with $n$ vertices. Applying the standard asymptotic expression \cite[Thm. 18.11]{MR2908619} yields
\[
Z_n \sim \spa(\mathbf{w}) (\phi(\tau) / (2 \pi \phi''(\tau)))^{1/2} (\tau / \phi(\tau))^{-n}  n^{-3/2}
\]
with $\tau \in ]0, \rho_\phi]$ the unique constant in that interval with $\psi(\tau)=1$ and $\spa(\mathbf{w})$ the greatest common divisor of all integers $k$ with $\omega_k \ne 0$. This proves Theorem~\ref{te:main4}.

By definition, the series $\phi(z)$ defined in the introduction and the series $\cR^s(z)$ agree. We may apply standard results on simply generated trees \cite[Ch. 8]{MR2908619} to see, that in our setting the tree $\cT_n$ is distributed like a conditioned Galton-Watson tree.
\begin{lemma}
\label{le:gwt}
Suppose that $\nu > 1$. Then there is a unique positive constant $\tau \in ]0, \rho_\phi]$ with $\psi(\tau) = 1$ and the simply generated tree $\cT_n$ is distributed like a Galton-Watson tree $\cT$ conditioned on having size $n$ with offspring distribution $\xi$ given by
\[
\Pr{\xi = k} = \omega_k \tau^k  / \phi(\tau) 
\]
Its first moment and variance are given by $\Ex{\xi}=1$ and $\sigma^2 = \tau \psi'(\tau)$. Moreover, $\xi$ has finite exponential moments, i.e. $\Ex{\exp(t\xi)} <0$ for all $t$ in some intervall $]-\delta, \delta[$ with $\delta >0$.
\end{lemma}

\subsection{The size-biased Galton-Watson tree}
\label{sec:sizegwt}
We define the {\em size-biased} distribution $\xi^*$ by 
\[
\Pr{\xi^* =k} = k \Pr{\xi = k}.
\]
For any integer $\ell \ge 0$, the size-biased Galton-Watson tree $\cT^{(\ell)}$ is a random plane tree together with a second or {\em outer} root having height $\ell$. It is defined in \cite[Ch. 3]{MR3077536} as follows. For $\ell = 0$, the tree $\cT^{(\ell)}$ is distributed like the $\xi$-Galton-Watson tree $\cT$ and the second root coincides with the first. For $\ell \ge 1$ there are two kinds of vertices, normal and mutant, and we start with a single mutant root. Mutant nodes have offspring according to independent copies of $\xi^*$. One of those is selected uniformly at random and declared its heir. If the heir has height strictly less than $\ell$, then it is also declared mutant. If it has height $\ell$, then it is declared the outer root, but remains normal. Normal vertices have offspring according to independent copies of $\xi$, all of whom are normal. The path connecting the two roots of the resulting tree is called its {\em spine}.

Note that for any mutant node, the probability that it has offspring of size $k$ and precisely the $i$th is selected as its heir, is given by $\Pr{\xi^* =k} /k = \Pr{\xi=k}$. Thus for any plane tree $T$ together with a vertex $v$ of $T$ having height $\ell$ it holds that
\begin{align}
\label{eq:gwtl}
\Pr{\cT^{(\ell)} = (T,v)} = \Pr{\cT = T}.
\end{align}
This equation is due to \cite[Eq. (3.2)]{MR3077536}.

\subsection{A deviation inequality}
\label{sec:deviation}
We will make use of the following well-known deviation inequality for one-dimensional random walk, found in most books on the subject.

\begin{lemma}
\label{le:deviation}
Let $(X_i)_{i \in \ndN}$ be an i.i.d. family of real-valued random variables with $\Ex{X_1} = 0$ and $\Ex{e^{t X_1}} < \infty$ for all $t$ in some open interval around zero. Then there are constants $\delta, c>0$ such that for all $n\in \ndN$, $x \ge 0$ and $0 \le\lambda\le\delta$ it holds that \[\Pr{|X_1 + \ldots + X_n| \ge x} \le 2 \exp(c n \lambda^2 - \lambda x).\]
\end{lemma}
The proof is by observing that $\Ex{e^{\lambda |X_1|}} \le 1 + c\lambda^2$ for some constant $c$ and sufficiently small~$\lambda$, and applying Markov's inequality to the random variable $\exp(\lambda(|X_1| + \ldots + |X_n|))$.

\subsection{The scaling limit and tail-bounds for the diameter}
We are now ready to prove Theorem~\ref{te:main2}. The idea will be to show that the $d_{\text{FPP}}$-distances in the random map $\mathbf{M}^s_n$ concentrate around a constant multiple of the distances of the decorated random plane tree $\cT_n$ of Lemma~\ref{le:sampler}. Using the convergence of $n^{-1/2} \cT_n$ to a multiple of the CRT we are going to deduce the scaling limit of 
$
(\mathbf{M}_n^s, \kappa n^{-1/2}d_{\text{FPP}})
$ for a suitable constant $\kappa$. Tail-bounds for the diameter of $\cT_n$ will be used to obtain tail-bounds for the diameter $\Di_{\text{FPP}}(\mathbf{M}_n^s)$.

\begin{proof}[Proof of Theorem~\ref{te:main2}]
Recall that by Lemma~\ref{le:gwt} we know that $\cT_n$ is distributed like a Galton-Watson tree conditioned on having $n$ vertices with an offspring distribution $\xi$ that is critical, not constant and has finite exponential moments. In the following, we let $\sigma^2$ denote its variance.

By Lemma~\ref{le:sampler}, the random map $\mathbf{M}_n^s$ is obtained from the plane tree $\cT_n$ by drawing for each vertex $v$ an $\cR^s$-structure $\beta_n(v)$ with size $d_{\cT_n}^+(v)$ uniformly at random and applying the bijection of Theorem~\ref{te:decomp}. The edges of $\mathbf{M}_n^s$ then correspond precisely to the edges of the $\cR^s$-structures. Let $\iota>0$ denote a random variable having finite exponential moments. We would like to assign an independent copy of $\iota$ to each edge of  $\mathbf{M}_n^s$ in order to form the first-passage percolation metric. To this end, note that for each rooted planar map there is a canonical linear order of its edges. For example, we could start with the root edge, continue in a counter-clockwise manner with the other edges adjacent to the root vertex, and then proceed likewise in a breadth-first-search manner. Hence we may form the first-passage percolation metric by taking an independent family $(\iota_i)_{i \in \ndN}$ and assigning $\iota_1, \iota_2, \ldots$ to the edges of $\mathbf{M}_n^s$ in that order until each edge of $\mathbf{M}_n^s$ has received a copy of $\iota$.

Likewise, we may form random maps (corresponding to decorated trees) by starting with other plane trees instead of $\cT_n$ and decorating its offspring sets with independent, uniformly drawn $\cR^s$-structures. Let $\cT$ denote a $\xi$-Galton-Watson tree and $\cT^{(\ell)}$, $\ell \ge 0$ the corresponding size-biased trees as described in Section~\ref{sec:sizegwt}. We are going to assume that all random objects considered are defined on the same probability space. We may form the decorated trees $(\cT, \beta)$ and $(\cT^{(\ell)}, \beta^{(\ell)})$. Hence $(\cT_n, \beta_n)$ is distributed like $(\cT, \beta)$ conditioned on $|\cT| = n$. Moreover, we may then form the first-passage percolation metric on the maps corresponding to these decorated trees by using the same family $(\iota_i)_{i \in \ndN}$ of independent copies of $\iota$, that we used to form the metric on the tree $(\cT_n, \beta_n)$.

Let $\eta$ be a random variable whose distribution is given by the first-passage percolation distance between the two spine-vertices in the map corresponding to $(\cT^{(1)}, \beta^{(1)})$ and set $\mu = \Ex{\eta}$. Given $\epsilon > 0$ let $\cE_1$ denote the event, that there exists a vertex $v$ of $\mathbf{M}_n^s$ having the property, that its tree height $\he_{\cT_n}(v)$ is at least $\log^2 n$ but the first-passage percolation distance from the root to $v$ in $(\mathbf{M}_n^s, d_{\text{FPP}})$ does not lie in the interval $(1 \pm \epsilon) \mu \he_{\cT_n}(v)$. Similarly, let $\cE_2$ denote the event that there exists a vertex $v$ of $\mathbf{M}_n^s$ such that $\he_{\cT_n}(v) \le \log^2 n$ and the $d_{\text{FPP}}$-distance from the root to $v$ is at least $\log^4 n$. We are going to show:

\begin{enumerate}
\item The probability, that $\cE_1$ or $\cE_2$ take place, tends to $0$ as $n$ becomes large.
\item There are constants $c,C>0$ such that for all $n$ and $x \ge 0$
\[
\Pr{\Di_{\text{FPP}}(\mathbf{M}_n^s) \ge x} \le C(\exp(-cx^2/n) + \exp(-cx)).
\]
\item There is a constant $C>0$ such that with probability tending to $1$ each maximal non-separable submap of $\mathbf{M}_n^s$ has first-passage percolation diameter at most $C \log n$.
\item For any $\epsilon >0$, it holds with probability tending to $1$ as $n$ becomes large, that for all vertices $u$ and $v$ of $\mathbf{M}_n^s$
\begin{align*}
|d_{\mathbf{M}_n}(u,v) - \mu d_{\cT_n}(u,v)| \le d_{\cT_n}(u,v)\epsilon + \log(n)^5.
\end{align*}
\item With respect to the Gromov-Hausdorff metric, we have that
\[
(\mathbf{M}_n^s, \frac{\sigma}{2\mu} n^{-1/2} d_{\text{FPP}}) \convdis (\CRT, d_{\CRT}).
\]
\end{enumerate}

We start with claim (1): Conditional on the family $(\iota_i)_{i \in \ndN}$, there is a finite set $\cH_1$ of $\cR^s$-decorated plane trees $(T, \gamma)$ with $n$ vertices such that $\cE_1$ takes place if and only if $(\cT_n, \beta_n) \in \cH_1$. Recall that $(\cT_n, \beta_n)$ is distributed like $(\cT, \beta)$ conditioned on having size $n$. Moreover,
\[
\Pr{|\cT| = n} \sim c n^{-3/2}
\]
for some constant $c>0$, since $\Ex{\xi} = 1$ and $\xi$ has finite variance. Hence
\[
\Pr{\cE_1 \mid (\iota_i)_i } = \Pr{(\cT_n, \beta_n) \in \cH_1 \mid (\iota_i)_i} = O(n^{3/2}) \Pr{(\cT, \beta) \in \cH_1 \mid (\iota_i)_i}.
\]
It follows from Equation~\eqref{eq:gwtl}, that for each $\cR^s$-decorated plane $(T, \gamma)$ and vertex $v$ of $T$ it holds that
\[
\Pr{(\cT^{(h_T(v))},\beta^{(h_T(v))}) = ((T,v), \gamma) } = \Pr{ (\cT, \beta) = (T, \gamma)}.
\]
Setting 
\[
\bar{\cH}_1 = \{((T, v_{(T, \gamma)}), \gamma) \mid (T, \gamma) \in \cH_1\}
\]
we obtain the bound
\[
\Pr{\cE_1 \mid (\iota_i)_i} = O(n^{3/2}) \sum_{\log^2 n \le \ell \le n} \Pr{ (\cT^{(\ell)}, \beta^{(\ell)}) \in \bar{\cH}_1 \mid (\iota_i)_i}.
\]
The event $(\cT^{(\ell)}, \beta^{(\ell)}) \in \bar{\cH}_1$ implies that the first-passage percolation distance between the ends of the spine in the map corresponding to $(\cT^{(\ell)}, \beta^{(\ell)})$ does not lie in the interval $(1 \pm \epsilon) \ell \mu$. But this distance is distributed like the sum $\eta_1 + \ldots \eta_\ell$ of independent copies of the random variable $\eta$. Hence we obtain that
\begin{align}
\label{eq:bound1}
\Pr{\cE_1} \le O(n^{3/2}) \sum_{\log^2 n \le \ell \le n} \Pr{\eta_1 + \ldots + \eta_\ell \notin (1 \pm \epsilon) \ell \mu}
\end{align}
The random variable $\eta$ has finite exponential moments: it is bounded by the first-passage percolation diameter of a random $\cR^s$-object whose number of vertices is given by the size-biased random variable $\xi^*$ with $\Pr{\xi^* =k} = k \Pr{\xi = k}$. The FPP-diameter of the $\cR^s$-object is bounded by the FPP-diameter of any fixed spanning tree and hence by the sum of $(\xi^* -1)$-many independent copies of $\iota$. Since $\xi$ (and hence $\xi^*$) and $\iota$ have finite exponential moments, so has $\eta$. Thus we may apply the deviation inequality from Lemma~\ref{le:deviation} and obtain that the bound in (\ref{eq:bound1}) converges to zero as $n$ becomes large. 

Analogously, we may repeat precisely the same arguments for the event $\cE_2$ in order to obtain
\begin{align}
\label{eq:bound2}
\Pr{\cE_2} \le O(n^{3/2}) \sum_{1 \le \ell \le \log^2 n} \Pr{\eta_1 + \ldots + \eta_\ell \ge \log^4 n}.
\end{align}
Again, this bound tends to $0$ by the deviation inequality from Lemma~\ref{le:deviation}.

We continue with claim (2):
It suffices to show such a bound for the height $\He_{\text{FPP}}(\mathbf{M}_n^s)$, i.e. the maximal $d_{\text{FPP}}$-distance of a vertex from the root vertex. Moreover, it suffices to consider parameters $x \ge \sqrt{n}$. With foresight, set $s = 1/(2\mu)$. By tail-bounds for conditioned Galton-Watson trees provided in \cite{MR3077536}, there are constants $C_1, c_1 >0$ (that do not depend on $n$ or $x$) such that 
\[
\Pr{\He(\cT_n) \ge sx} \le C_1 \exp(-c_1 x^2 / n).
\]
Hence
\[
\Pr{\He_{\text{FPP}}(\mathbf{M}_n^s) \ge x} \le C_1 \exp(-c_1 x^2 / n) + \Pr{\cE_3}
\]
with $\cE_3$ the event, that there exists a vertex $v$ in $\cT_n$ having tree-height $\he_{\cT_n}(x) \le sx$ but the first passage percolation distance from the root to $v$ in the map corresponding to $(\cT_n, \beta_n)$ is at least $x$. Using the same arguments as for the event $\cE_1$ and $\cE_2$, we obtain
\[
\Pr{\cE_3} \le O(n^{3/2}) \sum_{1 \le \ell \le \min(n,sx)} \Pr{\eta_1 + \ldots + \eta_\ell \ge x} \le O(n^{5/2}) \Pr{\eta_1 + \ldots + \eta_{\lfloor sx \rfloor} \ge x}.
\]
Recall that $1- s \mu >0$ by the choice of $s$. Applying the deviation inequality from Lemma~\ref{le:deviation} we obtain that there is a constant $c>0$ such that for all sufficiently small $\lambda > 0$
\[
\Pr{\cE_3} \le O(n^{5/2}) \exp(x(c s  \lambda^2 - \lambda (1 - s\mu))). 
\]
Taking $\lambda$ sufficiently small and using $x \ge \sqrt{n}$ we may bound this by $C_2 \exp(-c_2 x)$ for some constants $C_2, c_2 >0$ that do not depend on $n$ or $x$. Hence there are constants $C_3, c_3 > 0$ such that for all $x \ge \sqrt{n}$
\[
\Pr{\He_{\text{FPP}}(\mathbf{M}_n^s) \ge x} \le C_3(\exp(-c_3 x^2 / n) + \exp(-c_3 x)).
\]
We proceed with claim (3): It suffices to show such a bound for the $\cR^s$-structures. Clearly, for any $x \ge 0$
\begin{align*}
\Pr{\Di_{\text{FPP}}(\beta_n(v)) \ge x \text{ for some vertex $v$}} \le O(n^{3/2}) \Pr{\Di_{\text{FPP}}(\beta(v)) \ge x \text{ for some vertex $v$}}
\end{align*}
The diameters of the (maps corresponding to the) $\beta(v)$ are independent and identically distributed, hence we may bound this further by
\[
O(n^{5/2}) \Pr{\Di_{\text{FPP}}(\beta(o)) \ge x}
\]
with $o$ denoting the root vertex of the tree $\cT$. The diameter of (the map corresponding to) $\beta(o)$ is bounded by the diameter of any fixed spanning tree and hence by the sum of $(\xi-1)$-distributed many independent copies of the link-weight $\iota$. Since $\iota$ and $\xi$ have finite exponential moments, this bound converges to zero for $x = C \log n$ with $C$ a sufficiently large fixed constant.

We proceed with claim (4): Let $\epsilon >0$ be given. Let $u$ and $v$ be vertices of $\cT_n$ and let $w$ denote their lowest common ancestor. Then the tree distance between these vertices may be expressed by their heights
\begin{align}
\label{eq:t1}
d_{\cT_n}(u,v) = \he_{\cT_n}(u) + \he_{\cT_n}(v) - 2 \he_{\cT_n}(w).
\end{align}
Moreover, using the fact that maximal non-separable submaps intersect only at articulation points, we obtain that
\begin{align}
\label{eq:t2}
d_{\text{FPP}}(u,v) = \he_{\text{FPP}}(u) + \he_{\text{FPP}}(v) - 2 \he_{\text{FPP}}(w) - C(u,v)
\end{align}
with an error term $C(u,v) \ge 0$ that is bounded by the first-passage percolation diameter of some non-separable submap of $\mathbf{M}_n^s$.

By claim (3) we have with probability tending to $1$ that $C(u,v) \le C \log n$ regardless of the choice of vertices $u$ and $v$. Moreover, the complimentary event $\cE_2^c$ implies that for any vertex $x$ with $\he_{\cT_n}(x) \le \log^2 n$ it holds that $\he_{\text{FPP}}(x)\le \log^4 n$ and consequently
\[
|\he_{\text{FPP}}(x) - \mu \he_{\cT_n}(x)| \le \log^2 n + \log^4 n \le \epsilon \he_{\cT_n}(x) + 2 \log(n)^4.
\]
The complimentary event $\cE_1^c$ for $\epsilon' = \epsilon/\mu$ implies that, if $\he_{\cT_n}(x) \ge \log^2 n$, then
\[
|\he_{\text{FPP}}(x) - \mu \he_{\cT_n}(x)| \le \epsilon \he_{\cT_n}(x).
\]
It follows by claim (1) that  with probability tending to $1$ as $n$ becomes large that
\[
|d_{\text{FPP}}(u,v) - \mu d_{\cT_n}(u,v)| \le \epsilon d_{\cT_n}(u,v) + O(\log^4 n)
\]
for all vertices $u$ and $v$.

It remains to deduce claim (5).
By Lemma~\ref{le:gwt} we know that $\cT_n$ is distributed like a Galton-Watson tree conditioned on having $n$ vertices with an offspring distribution $\xi$ that is critical and has finite variance $\sigma^2$. Hence $n^{-1/2} \sigma \cT_n /2$ converges towards the continuum random tree. 

Claim (4) implies that for any $\epsilon>0$ it holds with probability tending to $1$ that
\[
d_{\text{GH}}((\mathbf{M}_n^s, n^{-1/2}d_{\text{FPP}}),(\cT_n, n^{-1/2} \mu d_{\cT_n})) \le \epsilon n^{-1/2} \Di(\cT_n) + o(1).
\] 
The rescaled diameter $n^{-1/2}\Di(\cT_n)$ converges weakly towards a multiple of the diameter of the CRT, which is almost surely finite. Hence
\[
d_{\text{GH}}((\mathbf{M}_n^s, n^{-1/2}d_{\text{FPP}}),(\cT_n, n^{-1/2} \mu d_{\cT_n})) \convp 0
\]
and consequently
\begin{align*}
(\mathbf{M}_n^s, \frac{\sigma}{2\mu}n^{-1/2}d_{\text{FPP}}) \convdis (\CRT, d_{\CRT}).
\end{align*}
\end{proof}

\section{Applications to the unrestricted and bipartite case}
We are going to check that simple outerplanar maps in both the unrestricted and bipartite case fall into our setting (with $\nu=\infty$) and calculate the scaling constants for the graph metric. In particular, we recover Theorem~\ref{te:main} as a consequence of Theorem~\ref{te:main2}. We also recover the asymptotic enumeration formula for the number of outerplanar maps and give a similar result for the bipartite case, proving Corollary~\ref{co:enum}.

\subsection{A formula for the scaling constant}

\begin{lemma}
\label{le:calc}
The scaling constant $\kappa$ in Theorem~\ref{te:main2} is given by $\kappa = \sigma/(2\mu')$ with $\sigma^2 = \tau \psi'(\tau)$ and $\mu'$ the average first-passage percolation distance between the $*$-vertex and the pointed vertex in a random $(\cC_*^{s})^\bullet$-object drawn according to a Boltzmann distribution with parameter $\tau$. That is, $(\cC_*^{s})^\bullet$-objects of the same size are equally likely and the probability generating function for the size is given by \[(\cC_*^{s})^\bullet(\tau z) / (\cC_*^{s})^\bullet(\tau) = z (\cC_*^{s})'(\tau z) / (\cC_*^{s})'(\tau).\]
\end{lemma}

\begin{proof}
We established in the proof of Theorem~\ref{te:main2} that 
\begin{align*}
(\mathbf{M}_n^s, \kappa n^{-1/2}d_{\text{FPP}}) \convdis (\CRT, d_{\CRT})
\end{align*}
with $\kappa = \sigma/(2\mu)$. Recall that $\sigma^2 = \tau \psi'(\tau)$ is as defined in Lemma~\ref{le:gwt} the variance of a certain offspring distribution $\xi$, with $\tau$ the unique number in $]0, \rho_\phi]$ satisfying $\psi(\tau)=1$. The constant $\mu$ was given by the first moment $\mu = \Ex{\eta}$ of a random distance $\eta$ defined as follows: Choose a size $r$ according to the size-biased offspring distribution $\xi^*$ and choose an $\cR^s$-structure $R$ with $r$ non-$*$-vertices uniformly at random. Glue the $\cC_*^{s}$-objects of that structure together at the $*$-vertices and let $\eta$ denote the first-passage percolation distance between the resulting unique $*$-vertex and a uniformly at random chosen non-$*$-vertex $u_R$.

We are going to argue that $\mu = \mu'$. The pointed object $(R, u_R)$ follows a Boltzmann-distribution for the class $(\cR^s)^\bullet$ with parameter $\tau$. That is, $(\cR^s)^\bullet$-objects of the same size are equally like and the size has probability generating function given by $(\cR^s)^\bullet(z\tau)/(\cR^s)^\bullet(\tau)$. The rules for the application of the pointing operator of combinatorial classes imply that the pointed class $(\cR^{s})^\bullet$ given by $(\cR^{s})^\bullet(z) = z (\cR^{s})'(z)$ may be decomposed into two factors
\[
(\cR^{s})^\bullet(z) = \Seq'(\cC_*^{s}(z)) (\cC_*^{s})^\bullet(z)
\]
with the second factor corresponding to the $\cC_*^{s}$-object containing the pointed vertex. Let $\tilde{C}$ denote the $C_*^{s}$-object of $R$ containing the pointed vertex $u_R$. The product rule for Boltzmann samplers \cite{MR2062483} implies that the random $(\cC_*^{s})^\bullet$-object $(\tilde{C}, u_R)$ also follows a Boltzmann distribution with parameter $\tau$. Hence $\eta$ is distributed like the $d_{\text{FPP}}$-distance from the $*$-vertex to the distinguished vertex in a Boltzmann distributed $(\cC_*^{s})^\bullet$-object and thus $\mu = \Ex{\eta} = \mu'$.
\end{proof}

\begin{figure}[ht]
	\centering
	\begin{minipage}{0.8\textwidth}
  		\centering
  		\includegraphics[width=0.6\textwidth]{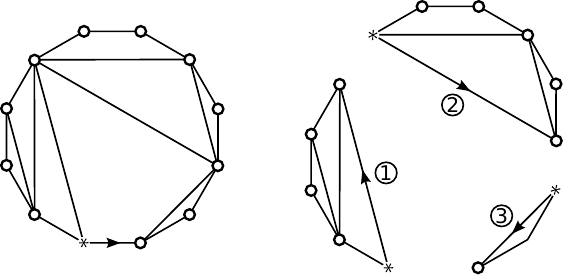}
  		\caption{The decomposition of edge-rooted dissections of polygons.}
  		\label{fi:diss}
	\end{minipage}
\end{figure}

\subsection{The class $\cM^{\text{out}}$ of all simple rooted outerplanar maps.}
\label{sec:unrestricted}
\subsubsection{Enumeration constants}
Recall that the class $\cC^{\text{out}}$ consists of edge-rooted dissections of polygons and the map consisting of two vertices connected by a root edge. By traversing the edges of the root face in clock-wise order, any edge-rooted dissection of a polygon may be decomposed into an ordered sequence of $\cC^{\text{out}}$-objects. This decomposition was previously established in \cite{MR2789731}. In order for the sizes to add up correctly, we require the root vertex to be replaced by a $*$-vertex that does not contribute to the total number of vertices. This yields a decomposition of the class $\cC^{\text{out}}_*$ illustrated in Figure~\ref{fi:diss}. In the language of analytic combinatorics, this may be expressed by the equation of generating series
\[
\cC_*^{\text{out}}(z) = z + \sum_{k=2}^\infty \cC_*^{\text{out}}(z)^k = z + \cC_*^{\text{out}}(z)^2 / (1 - \cC_*^{\text{out}}(z)).
\]
Solving for the series $\cC_*^{\text{out}}(z)$ yields
\[
\cC_*^{\text{out}}(z) = (1 + z - (z^2 - 6z + 1)^{1/2})/4.
\]
By Equation~\eqref{eq:co2} the series corresponding to the $\cR$-objects is given by 
\[
\cR(z) = 1 / (1- \cC_*^{\text{out}}(z)).
\]
For $\cM^s = \cM^{\text{out}}$, the generating function $\phi(x) = \cR(x)$  has positive radius of convergence $\rho_\phi = 3-2\sqrt{2}$. The unique solution of $\psi(\tau) = 1$ on the interval $[0, \rho_\phi]$ with $\psi(x) = x \phi'(x) / \phi(x)$ is given by $\tau = 1/6$. The sum $\nu = \psi(\rho_\phi) = \infty$ is infinite. Hence by Lemma~\ref{le:gwt} the offspring distribution $\xi$ has variance $\sigma^2 = \tau \psi'(\tau) = 18$. By Theorem~\ref{te:main4} we obtain that the number $|\cM^{\text{out}}_n|$ of planar maps in $\cM^{\text{out}}$ with $n$ vertices is asymptotically given by
\[
|\cM_n^{\text{out}}| \sim 8^n n^{-3/2} / (36 \sqrt{\pi}).
\]
This proves the first half of Corollary~\ref{co:enum}.

\subsubsection{The scaling constant}
In order to compute the scaling constant $\kappa = \sigma/2 \mu$, it remains to compute the stretch factor $\mu$. In Lemma~\ref{le:calc}, we identified $\mu =: \mu(\tau)$ as the average distance between the $*$-vertex and a uniformly at random drawn root in a random $\cC_*^{\text{out}}$-object that is Boltzmann distributed with parameter $\tau$. More generally, we may define the constant $\mu(y)$ for arbitrary parameters $y$ with $0<\cC_*^{\text{out}}(y) < \infty$. This constant was computed in \cite{2014arXiv1411.1865P} in order to compute the scaling constant of outerplanar graphs.

\begin{lemma}[{\cite[Lem. 8.9]{2014arXiv1411.1865P}}]
With $\w := \cC_*^\text{out}(y)$ it holds that \[\mu(y) = {\frac {8{\w}^{4}-16{\w}^{3}+4\w-1}{ \left( 4{
\w}^{3}-6{\w}^{2}-2\w+1 \right)  \left( 2\w-1
 \right) }}.\]
\end{lemma}
The parameters for outerplanar graphs are approximately given by $y \approx 0.17076$ and $\mu(y) \approx 5.46545$ \cite{2014arXiv1411.1865P}. In the setting for outerplanar maps, we have $y=\tau=1/6$ and therefore $w=1/4$ and  $\mu = \mu(\tau) = 7/3$.
Hence the convergence in Theorem~\ref{te:main2} now reads
\[
(\mathbf{M}_n, \frac{9}{7 \sqrt{2}} n^{-1/2}  d_{\mathbf{M}_n}) \convdis (\CRT, d_{\CRT}).
\]
We have thus recovered Theorem~\ref{te:main}.

\subsection{The class $\cM^{\text{bip}}$ of bipartite outerplanar maps}
\subsubsection{Enumeration constants}
We may treat this case in a very similar fashion. Let $\cM^{\text{bip}}$, $\cC^{\text{bip}}$ and $\cC_*^{\text{bip}}$ denote the corresponding bipartite versions of the classes. Any map from $\cC^{\text{bip}}$ is either a single edge with distinct ends or a dissection of a polygon in which each face has even degree. Using the decomposition illustrated in Figure~\ref{fi:diss}, we obtain that any $\cC_*^{\text{bip}}$-object either is the simple map with $2$ vertices or corresponds to a sequence of an uneven number of $\cC_*^{\text{bip}}$-objects. Hence
\[
\cC_*^{\text{bip}}(z) = z + \sum_{k=1}^\infty \cC_*^{\text{bip}}(z)^{2k +1} = z + \cC_*^{\text{bip}}(z)^3 /(1- \cC_*^{\text{bip}}(z)^2).
\]
We may solve this equation for $\cC_*^{\text{bip}}(z)$. Setting $\phi(z) =(1 -\cC_*^{\text{bip}}(z))^{-1}$and  $\psi(z) = z\phi'(z)/\phi(z)$ we obtain the parameter $\tau=-2+(4/3)\sqrt{3}$ as the unique solution of $\psi(\tau)=1$ in the interval $[0, \rho_\phi]$. The sum $\nu =\psi(\rho_\phi) = \infty$ is infinite. The variance of the offspring distribution $\xi$ is given by
\[\sigma^2 = \tau \psi'(\tau) = 9(\sqrt{3}-1).\]
Theorem~\ref{te:main4} yields that the number $|\cM^{\text{bip}}_n|$ of rooted bipartite simple outerplanar maps with $n$ vertices is asymptotically given by
\[
|\cM^{\text{bip}}_n| \sim {\frac { \left( -3+2\sqrt {3} \right) \sqrt {2}}{9\sqrt {\pi 
 \left( \sqrt {3}-1 \right) }}} (3 \sqrt{3}-5)^{-n} n^{-3/2}.
\]
This proves the second half of Corollary~\ref{co:enum}.
\subsubsection{The scaling constant}
In order to compute the scaling constant $\kappa = \sigma/(2 \mu)$, it remains to calculate the stretch factor $\mu$. This factor may be obtained by adapting the proof of \cite[Lem. 8.9]{2014arXiv1411.1865P}. We briefly sketch the calculation. We need to compute the expected distance from the $*$-vertex to the distinguished vertex of a random $(\cC_*^{\text{bip}})^\bullet$-object that follows a Boltzmann-distribution with parameter $\tau$. The class $(\cC_*^{\text{bip}})^\bullet$ is given by
\begin{align}
\label{eq:rooted}
(\cC_*^{\text{bip}})^\bullet(z) = z + (\cC_*^{\text{bip}})^\bullet(z) \sum_{k=1}^\infty (2k+1)(\cC_*^{\text{bip}}(z))^{2k}.
\end{align}
Hence by the construction-rules for Boltzmann samplers \cite{MR2062483,MR2095975}, the result of the following recursive procedure $\Gamma (\cC_*^{\text{bip}})^\bullet$ is distributed according to this distribution.
\begin{enumerate}[\qquad 1.]
\item Choose an integer $s \ge 0$ with distribution given by  \[\Pr{s=0} = \tau / (\cC_*^{\text{bip}})^\bullet(\tau)\] and, for each $k \ge 1$, 
\[
\Pr{s=2k} = (2k+1)(\cC_*^{\text{bip}}(\tau\textbf{}))^{2k}.
\]
\item If $s=0$ then return a single oriented root edge from a $*$-vertex to a root vertex. Otherwise, proceed with the following steps.
\item Let $C_1, \ldots, C_{2k}$ be independent Boltzmann-distributed $\cC_*^{\text{bip}}$-object with parameter $\tau$.
\item Let $(C,v)$ denote the result of recursively calling the sampler $\Gamma (\cC_*^{\text{bip}})^\bullet(\tau)$.
\item Choose an integer $0 \le i \le 2k$ uniformly at random and assemble an outerplanar map with root face degree 2k+2 from the ordered sequence of $\cC_*^{\text{bip}}$-objects \[(C_1, \ldots, C_{i}, C, C_{i+1}, \ldots, C_{2k})\] as illustrated in Figure~\ref{fi:diss}. Return this map rooted at the vertex $v$.
\end{enumerate}
For any integers $\ell, r \in \ndN_0$ with $\ell + r \in 2 \ndN$ we let $p_{\ell, r}$ denote the probability for the event $\cE_{\ell, r}$ that, in the above sampler, we have that $s=\ell + r$ and $i=\ell$. Hence
\begin{align}
\label{eq:00}
p_{\ell,r} = (\cC_*^{\text{bip}}(\tau))^{\ell+r}.
\end{align}
We let $L$ denote the distance from the $*$-vertex to the distinguished vertex $v$ and $R$ the distance from the other root edge vertex to $v$.
Moreover, we let $L_{\ell,r} = (L \mid\mid \cE_{\ell, r})$ and $R_{\ell,r}= (R \mid\mid \cE_{\ell, r})$ denote the corresponding variables conditioned on the event $\cE_{\ell, r}$.

Any shortest path from a root edge vertex to the distinguished vertex must traverse the boundary of the root face in one of the two directions, until it reaches the $(\cC_*^{\text{bip}})^\bullet$-object. Hence
\begin{align}
\label{eq:a1}
L_{\ell,r} &\eqdist \min\{\ell+L, 1+r+R\}, \\
\label{eq:a2}
R_{\ell,r} &\eqdist \min\{1+\ell+L, r+R\}.
\end{align}
Using Equation~\eqref{eq:rooted} we obtain 
\begin{align}
\label{eq:b1}
\Pr{s=0}=\tau / (\cC_*^\text{bip})^{\bullet}(\tau) = (1 - w^2 - w^3)/(1-w^2)
\end{align}
with $w := (\cC_*^\text{bip})(\tau)$. Using the fact that $|L-R| \le 1$, we may deduce from Equations \eqref{eq:00} to \eqref{eq:b1} that
\begin{align*}
\Ex{L} &= \sum_{\ell + r \in 2 \ndN} w^{\ell+r}(\one_{\{\ell \le r\}}(\ell + \Ex{L})  + \one_{\{\ell \ge r+2\}}(r+1 + \Ex{R})) + \frac{1 - w^2 - w^3}{1-w^2}, \\
\Ex{R} &= \sum_{\ell + r \in 2 \ndN} w^{\ell+r}(\one_{\{\ell \le r-2\}}(\ell+1 + \Ex{L})  + \one_{\{\ell \ge r\}}(r + \Ex{R})).
\end{align*}
This system of linear equations fortunately admits a unique solution, yielding
\[
\mu = \Ex{L} = {\frac {2{w}^{9}+2{w}^{8}-6{w}^{7}-8{w}^{6}+5{w}^{5}+7{w
}^{4}-{w}^{3}-4{w}^{2}+1}{4{w}^{8}-16{w}^{6}+19{w}^{4}-8{w}^
{2}+1}}= \frac{23}{8} - \frac{7}{24}\sqrt{3}.
\]
Letting $\mathbf{M}_n^{\text{bip}}$ denote the uniform simple rooted bipartite outerplanar map with $n$ vertices, we thus obtain
\[
 36{\frac {\sqrt {\sqrt {3}-1}}{69-7\sqrt {3}}} n^{-1/2} \mathbf{M}_n^{\text{bip}} \convdis \CRT.
\]
This concludes the proof of Theorem~\ref{te:main3}.

\section*{Acknowledgements}
I thank Nicolas Curien for discussions and feedback. I also thank an anonymous referee for helpful comments and suggestions.

\bibliographystyle{alpha}
\bibliography{subout}

\end{document}